\documentclass[reqno,12pt,twosided,a4paper]{amsart}
\usepackage{amssymb,amsfonts,latexsym,mathrsfs,amsmath}
\usepackage[all]{xy}
\usepackage{array}

\usepackage{amstext}
\usepackage{amssymb}
\usepackage{amsfonts}
\usepackage{amsthm}
\usepackage{amsopn}
\usepackage{amsbsy}
\usepackage{layout}

\numberwithin{equation}{section}

\newtheorem{theorem}{Theorem}[section]
\newtheorem{lemma}[theorem]{Lemma}
\newtheorem{proposition}[theorem]{Proposition}
\newtheorem{corollary}[theorem]{Corollary}

\theoremstyle{definition}
\newtheorem{definition}[theorem]{Definition}

\newtheorem{remark}[theorem]{Remark}

\newcommand{\Ga}{\Gamma}
\newcommand{\Z}{\mathbb{Z}}
\newcommand{\italic}{\textit}

\newcommand{\f}{\mathscr{F}}

\newcommand{\lhf}{\mathbf{LH}\f}

\newcommand{\ben}{\begin{enumerate}}
\newcommand{\een}{\end{enumerate}}

\newcommand{\mapsonto}{\twoheadrightarrow}
\newcommand{\cds}{\Theta^{\oplus}}
\begin{document}
\title[projective resolutions]{Projective resolutions for modules over infinite groups}
\author{Ehud Meir}
\date{6 October, 2010}
\maketitle
\begin{abstract}
We define a notion of complexity for modules over group rings of infinite groups.
This generalizes the notion of complexity for modules over group algebras of finite groups.
We show that if $M$ is a module over the group ring $kG$, where $k$ is any ring and $G$ is any group,
and $M$ has $f$-complexity (where $f$ is some complexity function) over some set of finite index subgroups of $G$,
then $M$ has $f$-complexity over $G$ (up to a direct summand). This generalizes the Alperin-Evens Theorem,
which states that if the group $G$ is finite then the complexity of $M$ over $G$ is the maximal complexity of $M$
over an elementary abelian subgroup of $G$. We also show how we can use this generalization in order to construct
projective resolutions for the integral special linear groups, $SL(n,\Z)$, where $n\geq 2$.
\end{abstract}

\begin{section}{Introduction}
Let $G$ be a finite group, let $p$ be a prime divisor of $|G|$, and let $k$ be a field of characteristic $p$.
Let $M$ be a finitely generated $kG$-module. By a theorem of Alperin and Evens (see \cite{AE}),
we know that the nonprojectivity of $M$ over $kG$ is determined by its nonprojectivity over $kE$,
where $E$ ranges over elementary abelian $p$-subgroups of $G$.

In order to state Alperin-Evens Theorem, we need the definition of complexity of a module.
If $M$ is a finitely generated $kG$-module, we say that a projective resolution $P^*\rightarrow M$ is minimal if $P^*$
is a direct summand of any other projective resolution $Q^*$ of $M$.
It follows that $rank_{kG}(P^n)\leq rank_{kG}(Q^n)$ for every $n$
(where by $rank_{kG}(M)$ we mean the minimal cardinality of a generating set of $M$ over $kG$).
In the case where $G$ is finite and $k$ is a field of prime characteristic,
every finitely generated module $M$ has a unique minimal projective resolution (see Chapter 2.4 of \cite{Evens} for a proof of this).

Let $P^*\rightarrow M$ be a minimal projective resolution of $M$. It is known that the sequence of numbers $a_n = rank_{kG}(P^n)$
has polynomial growth. We say that the complexity of $M$ is $c$, if the growth rate of the sequence $(a_n)$
is the same as the growth rate of the sequence $(n^{c-1})$.
We denote the complexity of $M$ by $c_{G}(M)$. The theorem of Alperin and Evens is the following:
\begin{theorem}\label{complexity} Let $G$ be a group, let $k$ be a field of characteristic $p$ and let $M$ be a $kG$-module.
Then \begin{equation}c_{G}(M)=max_E(c_{E}(M))\end{equation}
where the maximum is taken over all elementary abelian $p$-subgroups of $G$.\end{theorem}

Alperin and Evens proved the theorem in the following way: first, they reduce the general case to the case where $G$ is a $p$-group.
Then they use the fact that the complexity can be computed as the growth rate of the cohomology groups $H^*(G,M)$
(this follows from the fact that if $G$ is a $p$-group, then $kG$ has only one simple module, the trivial module $k$).
Then, in order to prove that the growth rate of the cohomology groups $H^*(G,M)$ is bounded by the growth rate over
the elementary abelian subgroups, they use a theorem of Serre, which states that a finite $p$-group $G$ is not elementary
abelian if and only if a certain product vanishes in the cohomology ring of $G$ (the constituents of this products are the Bocksteins of non trivial homomorphisms $G\rightarrow \Z_p$).

It is known that over a field of prime characteristic, a module $M$ has complexity 0 if and only if it is projective
(see Corollary 8.4.2 in \cite{Evens}). One of the results of Theorem \ref{complexity} is therefore Chouinard's theorem
(for the special case where $k$ is a field of prime characteristic), which states that $M$ is projective over $G$ if and only
if it is projective over each elementary abelian subgroup of $G$.

The theory of complexity was extended in \cite{bcr1} and in \cite{bcr2} by Benson, Carlson and Rickard to infinitely generated modules over $kG$,
where $G$ is a finite group and $k$ a field of prime characteristic. In \cite{bcr1} the authors proved that their definition of complexity
is equivalent to the following one: a module $M$ has complexity $c$ if and only if $M$ is a filtered colimit of finitely generated modules
of complexity $c$, but not a filtered colimit of finitely generated modules of complexity less than $c$.

In \cite{Ben2}, Benson used this theory in order to define complexity for $FP_{\infty}$ modules over $kG$,
where $k$ is a field of prime characteristic and $G$ is a group in Kropholler's hierarchy, $\lhf$
(an $FP_{\infty}$ module is a module which has a projective resolution which is finitely generated in each dimension.
An exposition of Kropholler's hierarchy can be found in \cite{krop}).
For a group $G \in\lhf$, and a module $M$ of type $FP_{\infty}$ over $G$, he proved that the set of
complexities of $M$ over finite subgroups of $G$ is bounded, and he defined the complexity of $M$ over $G$
to be the supremum of this set (the theory developed in \cite{bcr1} is needed here because the module $M$ might
not be finitely generated over the finite subgroups of $G$). He also gave an example of a module $M$ over the group
$G=\Z_2\times\Z_2\times\Z$ with a periodic resolution which has complexity 2
(and so, his definition of complexity does not agree with the definition of complexity as the minimal growth rate of a projective resolution).

The notion of complexity we will consider in this paper will be a generalization of the notion of complexity
for finite groups and will be based on growth rate of projective resolutions (and so will be different from the notion of complexity defined by Benson).

The idea in our computations will be the following: we will not consider growth rate of cohomology groups in order to study complexity of modules
(as in the proof of the Alperin-Evens Theorem). Instead, we will use a complex that was originally constructed by C. T. C. Wall
in order to construct projective resolutions explicitly, and show that their growth rate is less than or equal to a given function.
In order to do so, we will also use the same theorem of Serre that is used in the proof of the Alperin-Evens Theorem.
Wall's construction will enable us to consider arbitrary rings of coefficients and arbitrary groups, and not just fields and finite groups or groups in $\lhf$.

We will need two technical adjustments of the notions.
First, if $M$ is a $kG$-module where $k$ is any unital ring and $G$ is any group,
$M$ will not necessarily have a minimal projective resolution.
However, the fact that $M$ has \italic{some} projective resolution
with certain growth gives us some information about $M$.
We will therefore write $M\in\Theta_{kG}(f)$ and say that ``$M$ has $f$-complexity'' over $kG$
(where $f$ is some ``complexity function'', a notion which we will define in Section \ref{prelim})
if there is \italic{some} projective resolution $P^*$ of $M$ of growth rate $\leq f$.
In particular, the module under consideration will be an $FP_{\infty}$ module.
Notice that $M$ may have a variety of functions $f$ for which
$M\in\Theta_{kG}(f)$.  However, if $G$ is finite and $k$ is a field,
then $M$ will have complexity $c$ if and only if $M\in\Theta_{kG}(n^{c-1})$
and $M\notin\Theta_{kG}(n^{c-2})$.

Second, in most cases, we will be able to show that a module $M$ has $f$-complexity over $kG$ only up to a direct summand.
This means that we will only be able to show that there is a $kG$-module $N$ such that $M\oplus N$ has $f$-complexity over $kG$.
We will denote this situation by saying that $M$ has $f$-direct-summand-complexity (and write $M\in\Theta^{\oplus}_{kG}(f)$).
In case the group $G$ is finite and $k$ is a field,
it is possible to show that $M\in\Theta^{\oplus}_{kG}(f)$ if and only if $M\in\Theta_{kG}(f)$.
In general, it seems reasonable that the two conditions will be equivalent.
However, I do not know a proof of that.

In Section \ref{prelim} we will give the relevant definitions,
and prove some general facts about the classes $\Theta_{kG}(f)$ and $\cds_{kG}(f)$.
We will show that if $G$ is a group and $H$ a finite index normal subgroup,
then if $M\in\cds_{kT}(f)$ for every subgroup $H<T<G$ such that $T/H$ is a
$p$-Sylow subgroup for some prime $p$, then $M\in\cds_{kG}(f)$.
In Section \ref{wallconstruction} we will recall the construction of Wall's complex.
In Section \ref{elemab} we will consider the case where $G/H$ is a $p$-elementary abelian group.
We will show that under a certain assumption on the cohomology ring of $M$, $Ext^*_{kG}(M,M)$,
if $M\in\cds_{kA}(f)$ for some of the subgroups $H<A<G$ such that $A/H$ is a maximal proper subgroup of $G/H$,
then also $M\in\cds_{kG}(f)$. In Section \ref{consequences} we will show how we can use this result together with
Serre's Theorem and the construction from Section \ref{wallconstruction} in order to prove that if $M\in\cds_{kE}(f)$
for every subgroup $H<E<G$ such that $E/H$ is elementary abelian, then $M\in\cds_{kG}(f)$. By considering the case where
$G$ is finite and $k$ is a field of prime characteristic, we get Theorem \ref{complexity}. In Section \ref{SLN} we will present
an application to special linear groups over $\Z$. We will show that the trivial module $\Z$ satisfies $\Z\in\cds_{\Z SL(n,\Z)}(f)$
where $f(a)=a^{n-2}$, and we will also show that it is the best result possible, that is, we do not have a projective resolution
of polynomial growth rate of lower degree.
\end{section}

\begin{section}{Preliminaries}\label{prelim}
We would like to define properly the classes $\Theta_{kG}(f)$ and $\cds_{kG}(f)$. We begin with the following:
\begin{definition}\label{compfun} Let $f:\mathbb{N}\rightarrow\mathbb{R}_{+}$ be a function. We will say that $f$ is a proper complexity
function if the following condition holds:
There are two positive real numbers $c_1$ and $c_2$ such that $c_1<f(m+1)/f(m)<c_2$ for every $m\in\mathbb{N}$.
\end{definition}
For example, $n^a$, $log(n+1)$, $2^n$ and $2^{\sqrt{n}}$ are all complexity functions, while the function $n!$ is not.
The condition in the definition simply says that the growth rates of
$f(m)$ and of $f(m+1)$ are equal up to multiplication by some constant number.
\begin{remark}
In the fifth chapter of \cite{enc} an equivalence relation is defined on growth rates.
Two growth functions $f$ and $g$ are considered to be equivalent if and only if
there are constants $c_1,m_1,c_2,m_2$ such that $f(x)<c_1g(xm_1)$ and $g(x)<c_2f(xm_2)$.
We prefer not to consider this equivalence relation here because we would like
to distinguish, for example, the functions $2^{\sqrt{n}}$ and $3^{\sqrt{n}}$ (and also $2^n$ and $3^n$).
\end{remark}

For any ring $R$, the \italic{rank} of a finitely generated $R$-module $B$ is the minimal cardinality of a generating set of $B$.
We can now give our main definition:
\begin{definition}\label{complexitydef} Let $k$ be a ring, $G$ a group, and $M$ a $kG$-module.
Let $f$ be a proper complexity function. We say that $M$ has $f$-complexity
(and write $M\in\Theta_{kG}(f)$) if there is a projective resolution $P^*\rightarrow M\rightarrow 0$
and a number $d$ such that for $n\geq 0$ we have $rank_{kG}P^n\leq df(n)$.
We will say that $M$ has $f$-direct-summand-complexity and write
$M\in\cds_{kG}(f)$ in case there is a $kG$-module $N$ for which $M\oplus N\in\Theta_{kG}(f)$.\end{definition}
Some remarks are in order about this definition:
\begin{remark}
We do not use in the definition the fact that $f$ is a proper complexity function.
However, we need the properness assumption in most of what follows.
\end{remark}
\begin{remark}Notice in particular that if $M\in\Theta_{kG}(f)$ for some $k,G,M$ and $f$,
then $M$ has a projective resolution in which all the terms are finitely generated.
In other words, $M$ is an $FP_{\infty}$-module.
We will assume henceforth that all modules under consideration are $FP_{\infty}$ modules.
\end{remark}
\begin{remark}
If we have an onto map $(kG)^l\rightarrow (kG)^m$ for some $l$ and $m$ such that $l<m$, then
all finitely generated modules over $kG$ will have rank $\leq l$.
In that case, our discussion will still be valid but completely trivial.
By the right exactness of the tensor multiplication functor,
it is easy to see that this phenomenon does not occur, for example,
when there is a ring homomorphism from $k$ to some commutative ring or
to some skew field.
\end{remark}
\begin{remark}
Recall that if $G$ is a finite group and $k$ is a field,
then we say that a finitely generated $kG$-module $M$
has complexity $c$ if and only if its minimal projective resolution has growth rate $n^{c-1}$.
Since in this case any finitely generated module will have minimal projective resolution
of polynomial growth rate, it is easy to see that $M$ has complexity $c$
if and only if $M\in\Theta_{kG}(n^{c-1})$ and $M\notin\Theta_{kG}(n^{c-2})$.
In that sense, our discussion generalizes the notion of complexity, and for that reason we will be able to retrieve
Theorem \ref{complexity} as a special case of Theorem \ref{main3} (in order to do so, we will
also need to use Remarks \ref{remark1} and \ref{finiteindexremark}).
For modules over group rings of infinite groups, we do not necessarily have
a projective resolution of polynomial growth. See for example Theorem 2.6 of \cite{Stein}.
\end{remark}
\begin{remark}\label{remark1}
In case the group $G$ is finite and the ring $k$ is a field,
$M\in\cds_{kG}(f)$ if and only if $M\in\Theta_{kG}(f)$.
This is due to the existence of a minimal projective resolution for $M$ over $kG$.
It seems reasonable that this is true for any ring $k$ and any group $G$, but I do not know a proof of that.
\end{remark}
\begin{remark}
Suppose that $M$ is a $kG$-module and that $f$ is a proper complexity function.
If $M\in\Theta_{kG}(f)$, then we can use a projective resolution $P^*$ as in definition \ref{complexitydef}
in order to conclude that the $n$-th syzygy of $M$ has a resolution of growth rate $f(?+n)$.
The fact that $f$ is proper implies that every syzygy of $M$ is also in $\cds_{kG}(f)$.
It also implies that if the $n$-th syzygy of $M$ is in $\cds_{kG}(f)$,
then $M$ is also in $\cds_{kG}(f)$.
\end{remark}
\begin{remark}\label{finiteindexremark}
Notice that if $H$ is a subgroup of $G$ of finite index, then $M\in\Theta_{kG}(f)$
implies that $M\in\Theta_{kH}(f)$ (and similarly for $\cds$).
This is because a projective resolution for $M$ over $kG$ of growth rate $\leq f$
is also a projective resolution for $M$ over $kH$ of growth rate $\leq f$.
This property might fail if $H$ is not a finite index subgroup.
This happens for example in case $G=F$ is Thompson's group. It is known that $G$ is an $FP_{\infty}$ group,
(i.e. the trivial $\Z G$ module $\Z$ is $FP_{\infty}$) which has a subgroup $H$ which is
free abelian of infinite rank (in particular, $H$ is not finitely generated).
For more details on Thompson's group, see \cite{BrG}.
\end{remark}
\begin{remark}
 We can think about $\Theta_{kG}(f)$ as the class of all modules for which there exist a projective resolution
with growth rate bounded by $f$ (and similarly for $\cds_{kG}(f)$).
In this way the terminology $M\in\Theta_{kG}(f)$ makes sense.
\end{remark}

We prove now two general facts about complexity which we will need in the sequel.
\begin{lemma}\label{induction} Let $k,G,M,f$ be as above, and let $H$ be a subgroup of $G$.
If $M\in\Theta_{kH}(f)$, then $Ind_H^G(M)\in\Theta_{kG}(f)$\end{lemma}
\begin{proof}
Suppose that $P^*\rightarrow M$ is a projective resolution of $M$ over $kH$ which satisfies $rank_{kH}(P^n)\leq df(n)$.
Since the induction functor from $kH$ to $kG$ is exact and takes projective modules to projective modules,
$Ind_{kH}^{kG}(P^*)\rightarrow Ind_{kH}^{kG}(M)$ is a projective resolution of $Ind_{kH}^{kG}(M)$ over $kG$
which satisfies $rank_{kG}(Ind_{kH}^{kG}(P^n))\leq df(n)$. Therefore, $Ind_H^G(M)\in\Theta_{kG}(f)$.\end{proof}
\begin{remark} The lemma is also true if we replace $\Theta_{kG}$ by $\cds_{kG}$ and $\Theta_{kH}$ by $\cds_{kH}$.\end{remark}

\begin{lemma}\label{psylow} Let $H$ be a finite index normal subgroup of $G$.
Assume that $M\in\cds_{kS}(f)$ for every subgroup $H<S<G$ such that $S/H$ is a $p$-Sylow subgroup of $G/H$.
Then $M\in\cds_{kG}(f)$.\end{lemma}
\begin{proof}
For every prime divisor $p$ of $|G/H|$, let $H<S_p<G$ be a subgroup such that $S_p/H$
is a $p$-Sylow subgroup of $G/H$, and let $N_p$ be a module which satisfies $(M\oplus N_p)\in\Theta_{kS_p}(f)$.
Then $Ind_{S_p}^G(M\oplus N_p)\in\Theta_{kG}(f)$ for every $p\mid |G/H|$. Since $S_p$ has finite index in $G$,
we have a natural map $i_p:M\rightarrow Ind_{S_p}^G(M)$ given by $m\mapsto \sum_{g\in G/S_p}{g\otimes g^{-1}m}$.
The composition of this map with the natural map $q_p:Ind_{S_p}^G(M)\rightarrow M$ given by $g\otimes m\mapsto g\cdot m$ is multiplication by $|G/S_p|$.
Since the numbers $|G/S_p|$ are coprime, we see that the map
\begin{equation}\oplus_p Ind_{S_p}^G(M)\stackrel{\oplus q_p}{\rightarrow} M\end{equation}
splits. It follows that $M$ is a direct summand of $\oplus_p Ind_{S_p}^G(M)$,
and therefore also of $\oplus_p Ind_{S_p}^G(M\oplus N_p)$.
The last module is in $\Theta_{kG}(f)$, as it is a finite direct sum of modules in $\Theta_{kG}(f)$.
We therefore have $M\in\cds_{kG}(f)$ as desired.\end{proof}
\end{section}

\begin{section}{Wall's construction}\label{wallconstruction}
The complex of Wall (see \cite{Wall}) enables one to construct a resolution
for the trivial module $\Z$ over a group $G$ by using a resolution for the same module
over a normal subgroup $N$ of $G$ and over the quotient $G/N$.
We use here a variant of Wall's construction. For the reader's convenience, we give here the details of the construction.

Let $S$ be a ring, and let
\begin{equation}C=\cdots\rightarrow M^n\stackrel{g^n}{\rightarrow}M^{n-1}\rightarrow\cdots\rightarrow M^0\rightarrow 0\end{equation}
be a (finite or infinite) complex of $S$-modules. For every $n$, let
\begin{equation}\cdots F^{n,i}\stackrel{d^{n,i}}{\rightarrow}F^{n,i-1}\rightarrow\cdots\rightarrow F^{n,0}\rightarrow M^n\end{equation}
be a projective resolution of $M^n$. The idea of Wall's construction is that we can build from the projective resolutions a complex
$T^*$ of projective modules together with a map of complexes $T^*\rightarrow C^*$ which will induce an isomorphism in homology.
More precisely, we claim the following:

\begin{theorem}\label{Wall}(Wall)
Let $S$, $C^*$ and $F^{n,*}$ be complexes as described above.
Consider the graded module $T^n = \oplus_{i+j=n}F^{i,j}$.
There are maps on $T$, $d^{i,j}_k:F^{i,j}\rightarrow F^{i-k,j+k-1}$ for $k=0,1,\ldots$, such that: \\
1. The maps $d^n:T^n\rightarrow T^{n-1}$ given by $d^n = \sum_{k,i+j=n}d^{i,j}_k$ satisfy $d^{n+1}d^n=0$ for every $n$.
They therefore make $T^*$ a complex. \\
2. The map $\pi_n:T^n\mapsonto F^{n,0}\rightarrow M^n$ is a map of complexes $\pi:T^*\rightarrow C^*$
which induces an isomorphism in homology.\end{theorem}
\begin{proof} We will construct the differentials $d^{i,j}_k$ by induction on $k$.
We begin with $k=0$. In this case we need to give differentials $d^{i,j}_0: F^{i,j}\rightarrow F^{i,j-1}$.
These differentials would just be the differentials of the complexes $F^{i,*}$.
Notice that if we would have stopped here, Part 1 of Theorem \ref{Wall} would have held,
but Part 2 would have not (unless all the maps in $C^*$ are trivial). So we need to consider also the maps in $C^*$.

Consider now the case $k=1$. Using the Lifting Lemma (see Chapter 1.7 of \cite{Brown}),
we can lift the maps $g^n:M^n\rightarrow M^{n-1}$ to maps of complexes $F^i\rightarrow F^{i-1}$
which we shall denote by $g^{i,j}: F^{i,j}\rightarrow F^{i-1,j}$ (we use the fact that the modules $F^{i,j}$ are projective in order to apply the Lifting Lemma).
We can now introduce on $F^{*,*}$ differentials of bidegree $(-1,0)$ ($F^{*,*}$ is a bimodule in the obvious way).
These would just be $d^{i,j}_1=(-1)^jg^{i,j}$. We add the sign in order to make the equation $d^{i-1,j}_0d^{i,j}_1 + d^{i,j-1}_1d^{i,j}_0=0$ hold.

So far we have constructed a diagram which looks like the following figure:\\
\xymatrix{\label{figure1}
 & & & F^{2,2}\ar[d]^{d^{2,2}_1}\ar[r]^{d^{2,2}_0} & F^{2,1}\ar[d]^{d^{2,1}_1}\ar[r]^{d^{2,1}_0} & F^{2,0} \ar[d]^{d^{2,0}_1}\ar[r]^{d^{2,0}_0} & M^2\ar[d]^{g^2} \\
 & & & F^{1,2}\ar[d]^{d^{1,2}_1}\ar[r]^{d^{1,2}_0} & F^{1,1}\ar[d]^{d^{1,1}_1}\ar[r]^{d^{1,1}_0} & F^{1,0} \ar[d]^{d^{1,0}_1}\ar[r]^{d^{1,0}_0} & M^1\ar[d]^{g^1} \\
 & & & F^{0,2}\ar[r]^{d^{0,2}_0} & F^{0,1}\ar[r]^{d^{0,1}_0} & F^{0,0} \ar[r]^{d^{0,0}_0} & M^0}

If $d^{i-1,j}_1d^{i,j}_1=0$, we could have taken $d^{i,j}=d^{i,j}_0+d^{i,j}_1$,
and the construction of the differentials of $T^*$ would have been completed.
The problem is that the equations $d^{i-1,j}_1d^{i,j}_1=0$ might not hold.
We can now continue in the following way:
we consider the maps $d^{i-1,j}_1d^{i,j}_1$ as chain maps which lift the zero map,
and we use the Lifting Lemma to get maps $d^{i,j}_2:F^{i,j}\rightarrow F^{i-2,j+1}$
which will make another component of $d^2$ equal to zero.
By induction, at stage $k$ we add in this way another component $d^{i,j}_k$
which will make another component of $d^2$ equal to zero.

As the number of modules in each diagonal is finite, the sum $d^{i,j}=\sum_{k}d^{i,j}_k$ is finite.
Moreover, our construction yields that $d^n=\sum_{i+j=n}d^{i,j}$ is a differential on $T^*$,
and so we have Part 1 of the theorem. It is also easy to see that the map $\pi$ is a map of complexes.
The last thing we need to check is that $\pi$ induces an isomorphism in homology.
For that, we consider the filtration on $T^*$ by rows.
That is, we define \begin{equation}(L^kT)^n = \oplus_{i+j=n; i\leq k}F^{i,j}\end{equation}
and we consider the spectral sequence associated to this filtration.
As the rows of $F^{i,j}$ are exact complexes, it is easy to see that this spectral sequence
collapses at the first page. This implies that $\pi$ induces an isomorphism in homology as desired. \end{proof}

\begin{remark}
The original setting of Wall's complex was the following:
suppose that we have a short exact sequence of groups
\begin{equation}1\rightarrow N\rightarrow G\rightarrow G/N\rightarrow 1.\end{equation}
We would like to construct a free resolution for $\Z$ over $\Z G$
by using a free resolution $F$ for $\Z$ over $\Z N$ and a free resolution
$C$ for $\Z$ over $\Z [G/N]$. By inducing $F$ to $G$ we get a free resolution
for $\Z [G/N]$ over $\Z G$. By taking direct sums of $Ind_N^G(F)$,
we get a free resolution for every $\Z G$-module which is the inflation
of a free $\Z [G/N]$-module, and thus, we have a free resolution for every
module which appears in $C$.
We can now apply the construction to get a resolution for $\Z$ over $\Z G$.\end{remark}

We would like now to apply this construction in order to
prove a closure property of complexity of modules. We claim the following:
\begin{proposition}
Let $0\rightarrow M_1\rightarrow M_2\rightarrow M_3\rightarrow 0$
be a short exact sequence of $kG$-modules, and let $f$ be a proper
complexity function.
If two of the modules in the sequence are in $\Theta_{kG}(f)$,
then so is the third\end{proposition}
\begin{proof}
If $M_1$ and $M_3$ have resolutions with growth rate $\leq f$,
then the Horseshoe Lemma (see Lemma 2.2.8 in \cite{Weibel})
gives us a resolution of $M_2$ of growth rate $\leq f$.
If $M_1$ and $M_2$ have respective resolutions $P_1$ and $P_2$
of growth rate $\leq f$, we can consider the complex $M_1\rightarrow M_2$,
whose homology is $M_3$ in degree zero and 0 in all other degrees.
Using Wall's construction, we can construct a complex of projective modules
$T^*$ such that $T^n = P_2^n\oplus P_1^{n-1}$, and such that the homology of
$T^*$ is $M_3$ in degree zero and zero elsewhere. In other words, $T^*$ is a projective
resolution of $M_1$ whose growth rate is $\leq f$ (we need to use here the fact that $f$ is proper).
In a similar way, if $M_2$ and $M_3$ have complexities $\leq f$, we can consider the complex
$M_2\rightarrow M_3$. The homology of this complex is $M_1$ in degree 1 and zero elsewhere.
The corresponding Wall's complex would then be a complex of projective modules
$\cdots T_3\rightarrow T_2\rightarrow T_1\rightarrow T_0\rightarrow 0$ whose homology
is $M_1$ concentrated in degree 1. But this means that $T_1\rightarrow T_0$ is onto,
and since $T_0$ is projective, we have a decomposition $T_1 \cong T_1'\oplus T_0$.
In this way we get a complex $\cdots\rightarrow T_2\rightarrow T_1'\rightarrow M_1\rightarrow 0$
which is a projective resolution of $M_1$ of growth rate $\leq f$
(again- we need to use here the assumption that $f$ is a proper complexity function).\end{proof}
\end{section}

\begin{section}{The case of elementary abelian quotient}\label{elemab}
In this section we will prove the first induction result.
Let $k,G,M$ and $f$ be as in the previous sections.
We begin by recalling some facts about cohomology of finite groups and finite quotients.

Assume that $L<G$ is a normal subgroup of prime index $p$.
The group $G/L$ is a finite group of order $p$, and therefore we have a homomorphism
$G\rightarrow \Z_p$ with kernel $L$. This homomorphism corresponds to an element
$\zeta_L\in H^1(G,\Z_p)$, where the action of $G$ on $\Z_p$ is trivial.

By considering the connecting homomorphism $\delta$ which
corresponds to the short exact sequence of trivial $\Z G$-modules
\begin{equation}0\rightarrow \Z\stackrel{p}{\rightarrow}\Z\rightarrow\Z_p\rightarrow 0,\end{equation}
we get an element (the Bockstein) $\beta_L=\delta(\zeta_K)\in H^2(G,\Z)$.
By considering the description of cohomology classes as exact sequences,
it can be shown that the element $\beta_L$ corresponds to the exact sequence
\begin{equation}0\rightarrow \Z\stackrel{1\mapsto \sum_{i=0}^{p-1}x^iL}{\rightarrow}
\Z G/L\stackrel{L\mapsto (x-1)L}{\rightarrow }\Z G/L\stackrel{L\mapsto 1}{\rightarrow}\Z\rightarrow 0,\end{equation}
where $x$ is an element of $G$ such that $xL$ is a generator of $G/L$.
By tensoring the last sequence with $M$ over $\Z$, we get an exact sequence
\begin{equation}0\rightarrow M\rightarrow \Z G/L\otimes_{\Z} M\rightarrow\Z
G/L\otimes_{\Z} M\rightarrow M \rightarrow 0,\end{equation}
which corresponds to an element $\beta_L^M\in Ext^2_{kG}(M,M)$
(the sequence remains exact upon tensoring with $M$ since it splits over $\Z$).

Notice that we have a natural isomorphism
$Ind_L^G(M)\stackrel{\cong}{\rightarrow} \Z G/L\otimes_{\Z}M$ given by
$g\otimes m\mapsto gL\otimes g\cdot m$.
So the middle terms in the sequence which represent $\beta_L^M$ are isomorphic to
$Ind_L^G(M)$. Notice also that if $N$ is any $kL$-module,
then $\beta_L^M$ is cohomologically equivalent to the sequence
\begin{equation}\label{bocksteinit}0\rightarrow M\rightarrow Ind_L^G(N\oplus M)
\rightarrow Ind_L^G(N\oplus M)\rightarrow M\rightarrow 0\end{equation}
which is formed by taking the direct sum of the former sequence with the sequence
\begin{equation}0\rightarrow0\rightarrow Ind_L^G(N)\rightarrow Ind_L^G(N)
\rightarrow 0\rightarrow 0.\end{equation}
We will need to use the specific representation \ref{bocksteinit} of $\beta_L^M$.

We will now prove the main technical result of this paper.
In the next section we will use a theorem of Serre in order to apply this result to more concrete situations.
\begin{proposition}\label{resolutions main1}
Let $M$ be a $kG$-module, and let $f$ be a proper complexity function.
Assume that $G$ has normal subgroups $L_1,\ldots,L_m$ of index $p$ such that
$\beta_{L_1}^M\cdots\beta_{L_m}^M=0$ in $Ext^{2m}_{kG}(M,M)$.
If $M\in\cds_{kL_i}(f)$ for every $i$, then $M\in\cds_{kG}(f)$.\end{proposition}
\begin{proof}
We shall do the following: we shall represent $\beta_{L_1}^M\cdots\beta_{L_m}^M$
as an exact sequence, and than we shall use resolutions over the subgroups
$L_1,\ldots L_m$ and Wall's construction in order to create a projective complex of
growth rate $\leq f$ over this sequence. We than use the fact that the product of the
Bocksteins is zero in cohomology in order to derive a projective resolution of growth rate
$\leq f$ for $M\oplus N$, where $N$ is a module which will be described in the sequel.

By assumption, for each $i$ we have a $kL_i$-module $N_i$ and a projective resolution
\begin{equation}\cdots\rightarrow \widehat{F_i^1}\rightarrow
\widehat{F_i^0}\rightarrow M\oplus N_i\rightarrow 0\end{equation}
of growth rate $\leq f$. By inducing this sequence to $kG$ we get,
as in Lemma \ref{induction}, a projective resolution of $Ind_{L_i}^G(M\oplus N_i)$ of growth rate $\leq f$.

We denote the module $Ind_{L_i}^G(\widehat{F_i^n})$ by $F_i^n$.
By the discussion at the beginning of this section,
and by the fact that the cup product in cohomology corresponds to concatenation of
exact sequences, we see that the cohomology class of $\beta_{L_1}^M\cdots\beta_{L_m}^M$
can be represented by an exact sequence of the form \begin{equation}0\rightarrow M\rightarrow
Ind_{L_m}^G(M\oplus N_m)\rightarrow Ind_{L_m}^G(M\oplus N_m)\rightarrow\cdots\rightarrow\end{equation}
\begin{equation*}\rightarrow Ind_{L_1}^G(M\oplus N_1)\rightarrow
Ind_{L_1}^G(M\oplus N_1)\rightarrow M\rightarrow 0.\end{equation*}

Let $C^*$ be the complex obtained from this sequence by deleting the two copies of $M$
at the beginning and at the end. Thus, the zeroth homology group of $C^*$ is $M$,
and the $(2m-1)$-st homology group of $C^*$ is also $M$. All other homology groups of $C^*$ are trivial.

Every module in $C^*$ is of the form $Ind_{L_i}^G(M\oplus N_i)$ for some $i$,
and so for every module in $C^*$ we have a projective resolution of growth rate $\leq f$.
We can use Wall's construction in order to construct from these resolutions a complex $P^*$
together with a map $\pi:P^*\rightarrow C^*$ which induces isomorphism in homology.
For $l\geq 2m-1$, the module $P^{l}$ is the direct sum
\begin{equation}P^l=F_1^{l}\oplus F_1^{l-1}\oplus\ldots\oplus F_m^{l-2m+2}\oplus F_m^{l-2m+1}.\end{equation}

By considering the rank of the constituents of $P^{l}$, we see that the growth rate of $P^*$
is $\leq f$ (we use here the fact that $f$ is a proper complexity function.
We shall give after the proof a bound for the number of generators
in $P^*$ and in the resolution we will create from $P^*$).

We know that the product $\beta=\beta_{L_1}^M\cdots\beta_{L_m}^M$ is zero in $Ext^{2m}_{kG}(M,M)$.
We can interpret this fact in the following way: let us denote the kernel of
$P^{2m-1}\rightarrow P^{2m-2}$ by $Z^{2m-1}$, and the image of
$P^{2m}\rightarrow P^{2m-1}$ by $B^{2m-1}$. The map $\pi^{2m-1}:P^{2m-1}\rightarrow C^{2m-1}$
sends $Z^{2m-1}$ onto the image of $M\rightarrow C^{2m-1}$. Since
$\pi$ induces isomorphism in homology, the kernel of $res(\pi)|_{Z^{2m-1}}:Z^{2m-1}\rightarrow M$
is $B^{2m-1}$ and it induces an isomorphism $Z^{2m-1}/B^{2m-1}\cong M$.

Now, up to the $2m$-th term, $P^*$ is a projective resolution for $M$.
Therefore, the group $Ext^{2m}_{kG}(M,M)$ can be identified with the quotient of the abelian group
$Hom_{kG}(Z^{2m-1},M)$ by the subgroup which is the image of the restriction map $Hom_{kG}(P^{2m-1},M)\rightarrow Hom_{kG}(Z^{2m-1},M)$.
Via this identification, the cohomology class of $\beta$ is the class of the map $Z^{2m-1}\rightarrow Z^{2m-1}/B^{2m-1}\cong M$.
But since $\beta=0$, this means that the map $Z^{2m-1}\rightarrow M$ can be extended to a map
$P^{2m-1}\rightarrow M$, which implies that $P^{2m-1}/B^{2m-1}$ splits as $M\oplus N$,
where $N=P^{2m-1}/Z^{2m-1} = B^{2m-2}$.

This means that we have a resolution for $M\oplus N$ given by
\begin{equation}\cdots\rightarrow P^{2m}\rightarrow P^{2m-1}\rightarrow M\oplus N\rightarrow 0.\end{equation}
By the assumption that $f$ is a proper complexity function,
it is easy to see that the growth rate of this resolution is $\leq f$ as desired.\end{proof}
\begin{remark} In case the complexity function $f$ is exponential, the proposition is true
even without the assumption on the vanishing product in cohomology.
This is due to the following reason: if $G$ has a normal finite index subgroup $L$ of index $p$
then $G/L$ is cyclic and has a periodic resolution $C^*$. By tensoring $C^*$
with $M$, we get a resolution of $M$ by modules of the form $\Z G/L\otimes M \cong Ind_L^G(M)$.
If $M$ has a projective resolution $P^*$ over $L$ of growth rate $\leq f$, then
by Wall's construction, we get a resolution for $M$ over $kG$. The fact that the growth rate
of this resolution is $f$ again follows from the fact that if $f(n)=a^n$ for some $a>1$,
then there is a scalar $c>0$ such that $f(1)+f(2)+\ldots +f(n) < cf(n)$.
\end{remark}

Notice that the direct summand $N$ in the proof is actually the $2m-2$ syzygy of $M$.
Notice also that this construction gives us not only the asymptotic behavior of the
growth rate of the resolution, but also the explicit resolution.
If we denote the rank of $\widehat{F_i^n}$ by $d_i^n$, we see that for $l\geq 2m-1$
the rank of $P^l$ is (bounded by) $d_1^l + d_1^{l-1} + ... d_m^{l-2m+2} + d_m^{l-2m+1}$.
Therefore, the rank of the $n$-th term of our resolution (which is $P^{2m+n-1}$)
will be (bounded by)
$d_1^{n+2m-1} + d_1^{n+2m-2} + ... d_m^{n+1} + d_m^n$. Of course,
there might be a resolution for $M$ or for $M\oplus N$ with less generators.
\end{section}

\begin{section}{Consequences of proposition \ref{resolutions main1}}\label{consequences}
We would like now to prove our main result, using Proposition \ref{resolutions main1}.
We first recall the following theorem of Serre (see Theorem 6.4.1 in \cite{Evens})
\begin{theorem}(Serre) Let $G$ be a finite $p$-group which is not elementary abelian.
Then there are subgroups $L_1,\ldots,L_m$ of index $p$ in $G$ such that $\beta_{L_1}\cdots\beta_{L_m}=0$.\end{theorem}

\begin{proposition}\label{main2}Let $k$ be a ring, let $G$ be a group, and let $H$ be a normal subgroup of $G$
of index $p^l$ for some $l$. Let $M$ be a $kG$-module, and let $f$ be a complexity function.
Assume that for every subgroup $H<E<G$ for which $E/H$ is elementary abelian, $M\in\cds_{kE}(f)$.
Then $M\in\cds_{kG}(f)$.\end{proposition}
\begin{proof} We argue by induction on subgroups of $G$ which contain $H$.
If $G/H$ is elementary abelian, there is nothing to prove. Otherwise,
suppose that the result is true for every subgroup $H<L<G$ of index $p$ in $G$.
Since elementary abelian subgroups of $L/H$ are also elementary abelian subgroups of $G/H$,
we have by induction that $M\in\cds_{kL}(f)$ for every such subgroup $L$.
By Serre's Theorem, there are subgroups $L_1,\ldots,L_m$ of index $p$ such that
$\beta_{L_1}\cdots\beta_{L_m}=0$ (Just consider the non elementary abelian finite group $G/H$
and the fact that $\beta_L$ is $inf_{G/H}^G(\beta_{L/H})$).
This implies, by tensoring with $M$, that $\beta_{L_1}^M\cdots\beta_{L_m}^M=0$.
We can thus apply proposition \ref{resolutions main1} and conclude that $M\in\cds_{kG}(f)$.\end{proof}

In order to apply this to arbitrary finite quotients, we use Lemma \ref{psylow}:
\begin{proposition}\label{main3} Let $G$ be a group, $H$ a normal subgroup of finite index.
Let $M$ be a $kG$-module, and let $f$ be a proper complexity function.
Assume that for every subgroup $H<E<G$ for which $E/H$ is elementary abelian,
$M\in\cds_{kE}(f)$. Then $M\in\cds_{kG}(f)$.\end{proposition}
\begin{proof}
We already know that the proposition is true in case $G/H$ is a $p$-group.
For every prime number $p$ which divides $|G/H|$, let $H<S_p<G$ be a subgroup of
$G$ such that $S_p/H$ is a $p$-Sylow subgroup of $G/H$.
Using Proposition \ref{main2} together with the assumption,
we see that $M\in\cds_{kS_p}(f)$ for every $p$. Using Lemma \ref{psylow},
we conclude that $M\in\cds_{kG}(f)$.\end{proof}
If $p$ is a prime number which has an inverse in $k$,
we do not need to consider quotients which are $p$-groups. More precisely:
\begin{lemma} Assume that $|G/H|$ is invertible in $k$.
If $M$ is a $kG$-module such that $M\in\cds_{kH}(f)$ then $M\in\cds_{kG}(f)$.\end{lemma}
\begin{proof}
This follows from the fact that in case $|G/H|$ is invertible in $k$,
the natural (onto) map \begin{equation}Ind_H^G(M)\rightarrow M\end{equation}
\begin{equation*}g\otimes m\mapsto g\cdot m\end{equation*} splits by the map
\begin{equation}m\mapsto \frac{1}{|G/H|}\sum_{g\in G/H}g\otimes g^{-1}\cdot m.\end{equation}
By Lemma \ref{induction} we see that $M\in\cds_{kG}(f)$.\end{proof}
Proposition \ref{main3} together with the lemma above implies the following
\begin{corollary}\label{main4} Let $G$ be a group, $H$ a normal subgroup of finite index.
Let $M$ be a $kG$-module, and let $f$ be a proper complexity function.
Assume that for every subgroup $H<E<G$ for which $E/H$ is $p$-elementary abelian,
where $p$ is a prime number which is not invertible in $k$, we have $M\in\cds_{kE}(f)$.
Then $M\in\cds_{kG}(f)$.\end{corollary}

Consider now the special case where $M,G$ and $H$ are as before, and $M$ is projective
over every subgroup $H<E<G$ such that $E/H$ is elementary abelian.
We do not have here a proper complexity function, but it is easy to see that
by applying the same arguments from Propositions \ref{resolutions main1} and
\ref{main3} we conclude that $M$ has a projective resolution of finite length.
Since $M$ is projective over a finite index subgroup of $G$,
it is known that this implies that $M$ is projective over $G$.
Notice that this argument remain valid even in case $M$ is not finitely generated
(we can still use Wall's construction in order to derive a finite length projective resolution for $M$).
This gives us a proof of the following result of Aljadeff and Ginosar (see \cite{AG})
\begin{theorem} Let $k$ be a ring, $G$ a group, and $M$ a $kG$-module.
Assume that $H$ is a finite index normal subgroup of $G$, and that $M$ is projective over
every subgroup $H<E<G$ such that the quotient $E/H$ is elementary abelian.
Then $M$ is projective over $G$.\end{theorem}
\begin{remark} The theorem of Aljadeff and Ginosar is formulated
more generally for crossed product algebras.
The theorem we cite here is a direct consequence of their theorem.\end{remark}

We deduce one more corollary which we shall use in the next section.
\begin{corollary}\label{vfcd} Let $M$ be a $kG$-module. Assume that $G$ has a finite index normal
subgroup $H$ such that $M$ has a finitely generated projective resolution $P^*$ over $kH$.
If we denote by $d$ the largest rank of an elementary abelian $p$-subgroup of $G/H$,
where $p$ is a prime number which is not invertible in $k$, then $M\in\cds_{kG}(n^{d-1})$.\end{corollary}
\begin{proof} In view of corollary \ref{main4}, we only need to show that if $H<E<G$, and $E/H$ is
$p$-elementary abelian of rank $d$, then $M\in\cds_{kE}(n^{d-1})$.
This follows from the fact that we have a free resolution $P^*$ for $\Z$ over
$\Z [E/H]$ with growth rate $n^{d-1}$. By tensoring this resolution over $\Z$ with $M$,
we get a resolution $C^*$ for $M$ by modules which are direct sums of copies of the module
$\Z [E/H]\otimes M\cong Ind_H^E(M)$. Using the resolution $Ind_H^E(P^*)\rightarrow Ind_H^E(M)$ and the complex $C^*$,
we get by Wall's construction a projective resolution for $M$ over $kE$.
An easy computation shows that it has the desired growth rate.\end{proof}
\end{section}

\begin{section}{An application for special linear groups}\label{SLN}
In this section we show how one can construct projective resolutions of
polynomial growth for the group $G=SL(n,\Z)$ where $n\geq 2$.
We begin by recalling the definition of congruence subgroups.

Let $n,m\geq 2$ be two natural numbers.
We have a natural homomorphism of groups $\pi^n_m:SL(n,\Z)\rightarrow SL(n,\Z_m)$.
We denote the kernel of $\pi^n_m$ by $\Ga^n_m$.
The group $\Ga^n_m$ is called the principal congruence subgroup of level $m$.
It is known (see Exercise 3 in Chapter 2.4 of \cite{Brown})
that for $m>2$ the group $\Ga^n_m$ is torsion free.
It is also known that if $m>2$ then the $\Z\Ga^n_m$-module $\Z$
has a finite resolution by finitely generated free modules
(see Chapter 8.9 of \cite{Brown}). By using Corollary \ref{vfcd} we see that
$\Z\in\cds_{\Z G}(a^{d-1})$
where $d$ is the largest rank of an elementary abelian subgroup of $SL(n,\Z_m)$.

We will show here how we can get a slightly better result.
We will show that $\Z\in\cds_{\Z G}(f)$, where $f(a) = a^{n-2}$.
This means that we have a projective resolution $P^*$ of $\Z\oplus N$
over $\Z G$ such that $rank(P^a)\leq ta^{n-2}$ for some number $t$
and some $\Z G$-module $N$. The module $N$ will arise as a syzygy of
$\Z$ over $SL(n,\Z)$, and therefore will be torsion free over $\Z$.
Thus, if $k$ is any ring, we can tensor this resolution with $k$ over $\Z$
in order to obtain a projective resolution of $k\oplus (k\otimes N)$ over $kG$
of growth rate $\leq a^{n-2}$. It follows that $k\in\cds_{kG}(a^{n-2})$.

In order to construct our resolution we will do the following:
we will take the group $H=\Ga^n_{pq}$, where $p$ and $q$ are two
distinct odd primes, and we will prove that if $H<E<G$ is a subgroup such that
$E/H$ is elementary abelian, then $\Z$ has a projective resolution of growth rate
$\leq a^{n-1}$ over $E$. We then use Proposition \ref{main3}. Let $H_1=\Ga^n_p$ and
$H_2=\Ga^n_q$. We claim the following
\begin{lemma}\label{boundrank}
Let $r$ be a prime number different from $p$.
Every $r$-elementary abelian subgroup of $G/H_1=SL(n,\Z_p)$
has rank $\leq n-1$. A similar result holds for $H_2$.\end{lemma}
\begin{proof}
We can embed the group $SL(n,\Z_p)$ into $SL(n,F)$,
where $F$ is the algebraic closure of $\Z_p$.
It is known that any finite commutative subgroup of semisimple elements in $SL(n,F)$
is conjugate to a subgroup of the diagonal matrices
(matrices of order $r$ are semisimple in characteristic $p$.
This can be seen by considering their characteristic polynomial).
But it is easy to see that the subgroup of diagonal matrices
(which is isomorphic to $(F^*)^{n-1}$) does not have an $r$-elementary abelian group of rank $>n-1$.\end{proof}

This almost finishes the construction.
The only problem is that $SL(n,\Z_p)$ has an elementary abelian $p$-subgroups
of rank $\geq \frac{n^2-1}{4}$ (see \cite{thw}). On the other hand,
it follows from the lemma that every $p$-elementary abelian subgroup of
$SL(n,\Z_q)$ is of rank $\leq n-1$.
So we shall overcome this problem by considering $H$,
which is the intersection of $H_1$ and $H_2$.

We claim the following
\begin{lemma}\label{elementary sln} Denote by $\pi_{pq}:G\rightarrow G/H = SL(n,\Z_{pq})$
the natural projection. If $E$ is an elementary abelian subgroup of $SL(n,\Z_{pq})$,
and $\widehat{E}=\pi_{pq}^{-1}(E)$ then $\Z\in\Theta_{\Z \widehat{E}}(a^{n-2})$.\end{lemma}
\begin{proof}
First, notice that we have a natural isomorphism
\begin{equation} SL(n,\Z_{pq})\rightarrow SL(n,\Z_p)\times SL(n,\Z_q)\end{equation}
given by reduction mod $p$ and mod $q$
(the fact that this is indeed an isomorphism is an easy consequence of the Chinese Remainder Theorem).
Second, if $r$ is any prime number,
then any $r$-elementary abelian subgroup of $SL(n,\Z_{pq})$
is of the form $E_1\times E_2$, where $E_1$ is an $r$-elementary abelian subgroup of
$SL(n,\Z_p)$ and $E_2$ is an $r$-elementary abelian subgroup of $SL(n,\Z_q)$
(and we use the isomorphism above as identification).

Suppose now that $E<SL(n,\Z_{pq})$ is $r$-elementary abelian for some prime number $r$.
Then $E$ is of the form $E_1\times E_2$. The subgroup $E$ is contained in the subgroups
$K_1=E_1\times SL(n,\Z_q)$ and $K_2=SL(n,\Z_p)\times E_2$.
By Remark \ref{finiteindexremark} we see that it is enough to prove that $\Z$
has a projective resolution of growth rate $\leq a^{n-2}$ over
$\widehat{K_1}=\pi_{pq}^{-1}(K_1)$ or over $\widehat{K_2}=\pi_{pq}^{-1}(K_2)$.
The subgroup $\widehat{K_1}$ contains $H_1$ as a finite index normal subgroup,
and the quotient $\widehat{K_1}/H_1$ is isomorphic to $E_1$.
If $r\neq p$ we can use the fact that $H_1$ has a finite cohomological dimension over $\Z$,
and conclude by Corollary \ref{vfcd} and Lemma \ref{boundrank} that $\Z$ has a
$\Z \widehat{K_1}$-projective resolution of growth rate $\leq a^{n-2}$.
If $r=p$, we just consider instead the subgroups $\widehat{K_2}$ and $H_2$
and use the fact that $p\neq q$. This finishes the proof of the lemma.\end{proof}

The lemma above, together with Proposition \ref{main3} implies the following
\begin{proposition} Let $G=SL(n,\Z)$. Then $\Z\in\cds_{\Z G}(a^{n-2})$.\end{proposition}

We claim that $SL(n,\Z)$ does not have a projective resolution of lower growth rate. More precisely:
\begin{lemma} Let $G=SL(n,\Z)$. Assume that we have a $\Z G$-module $M$ and a projective resolution $P^*$
for $M\oplus \Z$ over $\Z G$. Then there is a constant $c>0$
such that $rank_{\Z G}(P^a) \geq ca^{n-2}$ for every $a$.\end{lemma}
\begin{proof} Consider the finite index congruence subgroup $\Ga^n_3$ and the quotient $SL(n,\Z_3)$.
Inside this quotient we have a 2- elementary abelian subgroup of rank $n-1$.
This is the subgroup $E$ which contains all matrices of the form
$diag((-1)^{e_1},\ldots,(-1)^{e_n})$ such that $\sum e_i$ $=$ $0$ $mod$ $2$.
Denote by $H$ the inverse image of $E$ inside $G$.
We thus have a short exact sequence
\begin{equation}1\rightarrow \Ga^n_3\rightarrow H\rightarrow E\rightarrow 1\end{equation}
We claim that this sequence splits. Indeed, the subgroup of $H$ with the same description
(all matrices of the form  $diag((-1)^{e_1},\ldots,(-1)^{e_n})$ such that $\sum e_i$ $=$ $0$ $mod$ $2$)
maps isomorphically onto $E$. This means in particular that the inflation map $H^*(E,\Z)\rightarrow H^*(H,\Z)$
is one to one.
Since the rank of abelian groups is monotonously increasing, we have
\begin{equation} rank_{\Z}(H^a(E,\Z))\leq rank_{\Z}(H^a(H,\Z))= rank_{\Z}(Ext_{\Z H}^a(\Z,\Z))\end{equation}
\begin{equation*}\leq rank_{\Z}(Ext_{\Z H}^a(\Z\oplus M,\Z))\leq rank_{\Z}(Hom_{\Z H}(P^a,\Z))\end{equation*}
\begin{equation*}\leq rank_{\Z H}(P^a)\leq |G/H| rank_{\Z G}(P^a)\end{equation*}
But the rank of $H^a(E,\Z)$ is bounded from below by $\frac{a^{n-2}}{(n-2)!}$
(this is because the structure of the cohomology ring is known- it is a polynomial ring generate by $n-1$ variables in degree 1).
We conclude that $\frac{a^{n-2}}{(n-2)!|G/H|} \leq rank_{\Z G}(P^a)$ as desired. \end{proof}
\begin{remark} We could have use, of course, $\Ga^n_p$ for any odd prime $p$.
The choice of 3 was arbitrary.\end{remark}
\end{section}

\end{document}